\documentclass[12pt]{article}
\usepackage{amsfonts}
\usepackage{amssymb}
\usepackage{amsmath}

\newtheorem{theorem}{Theorem}[section]

\newtheorem{corollary}{Corollary}[section]

\newtheorem{definition}{Definition}[section]
\newtheorem{example}{Example}[section]

\newtheorem{lemma}{Lemma}[section]

\newtheorem{remark}{Remark}[section]

\newenvironment{proof}[1][Proof]{\noindent\textbf{#1.} }{\ \rule{0.5em}{0.5em}}
\begin{document}

\title{Centrosymmetric nonnegative realization of spectra\thanks{
Supported by CONICYT-FONDECYT 1170313, Chile; CONICYT-PAI 79160002, 2016,
Chile.}}
\date{}
\author{Ana I. Julio, Oscar Rojo, Ricardo L. Soto\thanks{Corresponding author, ajulio@ucn.cl, (Ana I Julio), orojo@ucn.cl (Oscar Rojo), rsoto@ucn.cl (Ricardo L.
Soto)} \\
{\small Departamento de Matem\'{a}ticas, Universidad Cat\'{o}lica del Norte}%
\\
{\small Casilla 1280, Antofagasta, Chile.}}
\maketitle

\begin{abstract}
A list $\Lambda =\{\lambda _{1},\lambda _{2},\ldots ,\lambda _{n}\}$ of
complex numbers is said to be realizable if it is the spectrum of an
entrywise nonnegative matrix. In this paper we intent to characterize those
lists of complex numbers, which are realizable by a centrosymmetric
nonnegative matrix. In particular, we show that lists of nonnegative real
numbers, and lists of complex numbers of Suleimanova type (except in one
particular case), are always the spectrum of some centrosymmetric nonnegative matrix. For the general lists we give sufficient conditions via a perturbation result. We also show that for $n=4,$ every realizable list of
real numbers is also realizable by a nonnegative centrosymmetric matrix.
\end{abstract}




\section{Introduction}
The \textit{nonnegative inverse eigenvalue problem }(hereafter NIEP) is the
problem of finding necessary and suffcient conditions for the existence of a 
$n\times n$ entrywise nonnegative matrix with prescribed complex spectrum $%
\Lambda =\{\lambda _{1},\lambda _{2},\ldots ,\lambda _{n}\}.$ If there
exists an $n\times n$ nonnegative matrix $A$ with spectrum $\Lambda $ we say
that $\Lambda $ is realizable and that $A$ is the realizing matrix. A
complete solution for the NIEP is known only for $n\leq 4,$ which shows the
difficulty of the problem. Throughout this paper, if $\Lambda =\{\lambda
_{1},\lambda _{2},\ldots ,\lambda _{n}\}$ is realizable by a nonnegative
matrix $A$, then $\lambda _{1}=\rho (A)=\max \{\left\vert \lambda
_{i}\right\vert ,$ $\lambda _{i}\in \Lambda \}$ is the Perron eigenvalue of $%
A.$ We shall denote the transpose of a matrix $A$ by $A^{T}$, $\lfloor
x\rfloor $ and $\lceil x\rceil $ denote the largest integer less than or
equal to $x$ and the smallest integer greater than or equal to $x,$
respectively. $J$ will denote the \textit{counteridentity} matrix, that is, $%
J=[\mathbf{e}_{n}\mid \mathbf{e}_{n-1}\mid \cdots \mid \mathbf{e}_{1}]$.
Then it is clear that $J^{T}=J,J^{2}=I$. \newline
Observe that multiplying a matrix $A$ by $J$ from the left results in
reversing its rows, while multiplying $A$ by $J$ from the right results in
reversing its columns. 
A vector $x$ is called \textit{symmetric} if $Jx=x$. \\ \\
The set of all $n\times n$ real matrices with constant row sums equal to $%
\alpha \in 
\mathbb{R}
$ will be denote by $\mathcal{CS}_{\alpha }.$ It is clear that $\mathbf{e}%
=[1,1,\ldots ,1]^{T}$ is an eigenvector of any matrix in $\mathcal{CS}%
_{\alpha },$ corresponding to the eigenvalue $\alpha .$ The
relevance of matrices with constant row sums is due to the well known fact
that the problem of finding a nonnegative matrix $A$ with spectrum $\Lambda
=\{\lambda _{1},\ldots ,\lambda _{n}\}$ is equivalent to the problem of
finding a nonnegative matrix $B\in \mathcal{CS}_{\lambda _{1}}$ with
spectrum $\Lambda ,$ that is, $A$ and $B$ are similar if the Perron
eigenvalue is simple and they are cospectral otherwise (see \cite{Julio2}).
\\ \\
In this paper we study the NIEP for centrosymmetric matrices.
Centrosymmetric matrices appear in many areas: physics, communication
theory, differential equations, numerical analysis, engineering, statistics,
etc. Now we state the definition and certain properties about
centrosymmetric matrices.
\begin{definition}
A matrix $C=(c_{ij})\in M_{m,n}$ is said to be centrosymmetric, if its
entries satisfy the relation 
\begin{equation*}
c_{i,j}=c_{m-i+1,n-j+1}.
\end{equation*}
\end{definition}\ \\
The definition means that a centrosymmetric matrix $C$ can be written as $%
J_{m}CJ_{n}=C,$ where $J_{n}$ is the $n\times n$ counteridentity matrix.
That is, $J_{n}=\left[ \mathbf{e}_{n}\mid \mathbf{e}_{n-1}\mid \cdots \mid 
\mathbf{e}_{1}\right] .$ For $m=n,$ we shall write $J$ instead $J_{n}.$%
\newline
\ \newline
The following properties and results on centrosymmetric matrices are easy to
verify (see \cite{Cantoni}).

\begin{lemma}
\cite{Cantoni}\label{Le1} Let $C_{1}$ and $C_{2}$ be centrosymmetric
matrices, and $\alpha _{1},\alpha _{2}\in 
\mathbb{R}
.$ Then, \newline
$i)$ $C_{1}^{-1}$, if exists, $ii)$ $C_{1}^{T}$, 
$iii)$ $\alpha _{1}C_{1}\pm \alpha _{2}C_{2}$, 
$iv)$ $C_{1}C_{2}$ are all centrosymmetric.
\end{lemma}

\begin{lemma}
\cite{Cantoni}\label{eigen} Let $C$ be an $n\times n$ centrosymmetric nonnegative matrix. If $\mathbf{v}$ is an eigenvector of $C$
corresponding to the eigenvalue $\lambda ,$ then $J\mathbf{v}$ is also an
eigenvector of $C$ corresponding to $\lambda $. Moreover, if $\lambda _{1}$
is the Perron eigenvalue of $C$, then there is a nonnegative eigenvector $%
\mathbf{x}$ such that $J\mathbf{x}=\mathbf{x}$.
\end{lemma}

\begin{theorem}
\cite{Cantoni}\label{Cen 1} Let $C$ be an $n\times n$ centrosymmetric
matrix. \newline
$i)$ If $n=2m,$ then $C$ can be written as 
\begin{equation*}
C=%
\begin{bmatrix}
A & JBJ \\ 
B & JAJ%
\end{bmatrix}%
,
\end{equation*}%
where $A,B$ and $J$ are $m\times m$ matrices. \newline
$ii)$ If $n=2m+1,$ then $C$ can be written as 
\begin{equation*}
C=%
\begin{bmatrix}
A & \mathbf{x} & JBJ \\ 
\mathbf{y}^{T} & c & \mathbf{y}^{T}J \\ 
B & J\mathbf{x} & JAJ%
\end{bmatrix}%
,
\end{equation*}%
where $A,B$ and $J$ are $m\times m$ matrices, $\mathbf{x}$ and $\mathbf{y}$
are $m$-dimensional vectors, and $c$ is a real number.
\end{theorem}

\begin{theorem}
\cite{Cantoni}\label{Cen 2} Let $C$ be an $n\times n$ centrosymmetric
matrix. \newline
$i)$ If $C=
\begin{bmatrix}
A & JBJ \\ 
B & JAJ%
\end{bmatrix}%
$, then $C$ is orthogonally similar to the matrix 
\begin{equation*}
\begin{bmatrix}
A+JB &  \\ 
& A-JB%
\end{bmatrix}%
.
\end{equation*}
Moreover, if $C$ is nonnegative
with the Perron eigenvalue $\lambda _{1}$, then $\lambda _{1}$ is the Perron
eigenvalue of $A+JB$. \newline
$ii)$ If $C$ is written as $C=%
\begin{bmatrix}
A & \mathbf{x} & JBJ \\ 
\mathbf{y}^{T} & c & \mathbf{y}^{T}J \\ 
B & J\mathbf{x} & JAJ%
\end{bmatrix}%
$, then $C$ is orthogonally similar to the matrices $%
\begin{bmatrix}
c & \sqrt{2}\mathbf{y}^{T} &  \\ 
\sqrt{2}\mathbf{x} & A+JB &  \\ 
&  & A-JB%
\end{bmatrix}%
$ and \newline
$%
\begin{bmatrix}
A+JB & \sqrt{2}\mathbf{x} &  \\ 
\sqrt{2}\mathbf{y}^{T} & c &  \\ 
&  & A-JB%
\end{bmatrix}%
$. Moreover, if $C$ is a
nonnegative matrix with the Perron eigenvalue $\lambda _{1}$, then $\lambda
_{1}$ is the Perron eigenvalue of $\begin{bmatrix}c & \sqrt{2}\mathbf{y}^{T}\\ \sqrt{2}\mathbf{x} & A+JB\end{bmatrix}$.
\end{theorem}\ \\
The following theorem, due to R. Rado and published by H. Perfect in \cite%
{Perfect1}, will be used to throughout the paper to show some of our
results. Rado's theorem show how to modify $r$ eigenvalues of an $n\times n$
matrix, $r<n,$ via a rank-$r$ perturbation, without changing any of the
remaining $n-r$ eigenvalues (see \cite{Julio, Soto4} for the way in which
Rado's result has been applied to the NIEP). The case $r=1,$ is the well
known Brauer's theorem \cite{Brauer}, which states that 
\begin{align}
\left. 
\begin{array}{c}
\text{If }A\text{ has eigenvalues }\lambda _{1},\lambda _{2},\ldots ,\lambda
_{n},\text{ then }A+\mathbf{vq}^{T}, \\ 
\text{has eigenvalues }\lambda _{1},\ldots ,\lambda _{k-1},\lambda _{k}+%
\mathbf{v}^{T}\mathbf{q},\lambda _{k+1},\ldots ,\lambda _{n}, \\ 
\text{where }A\mathbf{v}=\lambda _{k}\mathbf{v}\text{ \ and \ }\mathbf{q\in 
\mathbb{C}
}^{n}.\text{\ }%
\end{array}%
\right\}  \label{Bra}
\end{align}

\begin{theorem}
Rado \cite{Perfect1}\label{Rado} Let $M$ be an $n\times n$ arbitrary matrix
with spectrum $\Lambda =\{\lambda _{1},\lambda _{2},\ldots ,\lambda _{n}\}.$
Let $X=\left[ \mathbf{x}_{1}\mid \cdots \mid \mathbf{x}_{r}\right] $ be such
that $rank(X)=r$ and $M\mathbf{x}_{i}=\lambda _{i}\mathbf{x}_{i},$ $%
i=1,\ldots ,r,$ $r<n.$ Let $\mathcal{C}$ be an $r\times n$ arbitrary matrix.
Then $M+X\mathcal{C}$ has eigenvalues $\mu _{1},\ldots ,\mu _{r},\lambda
_{r+1},\ldots ,\lambda _{n},$ where $\mu _{1},\ldots ,\mu _{r}$ are
eigenvalues of the matrix $\Omega +\mathcal{C}X$ with $\Omega =diag\{\lambda
_{1},\ldots ,\lambda _{r}\}.$
\end{theorem}\ \\
In \cite{SJC} the authors prove that a realizable list of complex numbers of
Suleimanova type $\Lambda =\{\lambda _{1},\lambda _{2},\ldots ,\lambda
_{n}\},$ that is, with $\lambda _{1}$ being the Perron eigenvalue, and 
\begin{equation*}
\lambda _{i}\in \mathcal{F}=\{\lambda _{i}\in \mathbb{C}:Re\lambda _{i}\leq
0,|Re\lambda _{i}|\geq |Im\lambda _{i}|\},\text{ }i=2,\ldots ,n,
\end{equation*}%
is in particular realizable with spectrum $\Lambda $ and with prescribed
diagonal entries $\omega _{1},\omega _{2},\ldots ,\omega _{n}$ if and only
if $\sum\limits_{i=1}^{n}\omega _{i}=\sum\limits_{i=1}^{n}\lambda _{i}.$
This result, which we shall use in Section $3,$ follows directly from the
following lemma and (\ref{Bra}), which we set here for sake of completeness:
\begin{lemma}
\label{SJC}Let $\Lambda ^{\prime }=\{-\sum\limits_{i=2}^{n}\lambda
_{i},\lambda _{2},\ldots ,\lambda _{n}\}$ be a realizable list of complex
numbers, with Perron eigenvalue $-\sum\limits_{i=2}^{n}\lambda _{i}.$ Then
for any $\lambda _{1}\geq -\sum\limits_{i=2}^{n}\lambda _{i},$ the list $%
\Lambda =\{\lambda _{1},\lambda _{2},\ldots ,\lambda _{n}\}$ is the spectrum
of a nonnegative matrix with prescribed diagonal entries $\Gamma =\{\omega _{1},\omega _{2},\ldots ,\omega _{n}\},$ if and only if $\sum%
\limits_{i=1}^{n}\omega _{i}=\sum\limits_{i=1}^{n}\lambda _{i}.$
\end{lemma}

\begin{proof}
Since $\Lambda ^{\prime }$ is realizable, there exists a nonnegative matrix $%
B\in \mathcal{CS}_{\beta },$ where $\beta =-\sum\limits_{i=2}^{n}\lambda
_{i},$ with spectrum $\Lambda ^{\prime }$ and $tr(B)=0.$ Let $\mathbf{q 
}^{T}=\left[\omega _{1},\ldots ,\omega _{n}\right] .$ Then from (\ref{Bra}%
), $A=B+\mathbf{eq }^{T}$ is a nonnegative matrix with spectrum $%
\Lambda $ and with diagonal entries $\omega _{1},\omega _{2},\ldots ,\omega _{n}$.
\end{proof}

\section{Real centrosymmetric matrices with prescribed spectrum}

In this section we show that a list of complex numbers $\Lambda =\{\lambda
_{1},\lambda _{2},\ldots ,\lambda _{n}\},$ with $\Lambda =\overline{\Lambda }
$, is always the spectrum of a real centrosymmetric matrix \textbf{(}not
necessarily nonnegative\textbf{)}. That is, the real centrosymmetric inverse
eigenvalue problem for a list of conjugate complex numbers has always a
solution.

\begin{theorem}
Let $\Lambda=\{\lambda_{1},\lambda_{2},\ldots,\lambda_{n}\}$ be a list of
complex numbers, with $\Lambda=\overline{\Lambda}$, $n\geq3$. Then $\Lambda$
is the spectrum of an $n\times n$ real centrosymmetric matrix.
\end{theorem}

\begin{proof}
We shall distinguish two cases:\newline
\newline
Case 1. Let $n$ be even. First we consider $\Lambda $ with only two
real numbers, that is, 
\begin{equation*}
\Lambda =\{\lambda _{1},\lambda _{2},z_{1},\overline{z}_{1},\ldots ,z_{m},%
\overline{z}_{m}\},\ \ z_{j}=a_{j}+ib_{j},\ \ a_{j}\in \mathbb{R},\ \ b_{j}>0,\
j=1,\ldots ,m.
\end{equation*}%
If $m$ is even, we take the partition $\Lambda =\Lambda _{1}\cup \Lambda _{2}
$ with 
\begin{equation*}
\Lambda _{k}=\{\lambda _{k1},z_{k1},\overline{z}_{k1},\ldots ,z_{k\frac{m}{2}%
},\overline{z}_{k\frac{m}{2}}\},\ k=1,2,\ \ \lambda _{11}=\lambda _{1},\
\lambda _{21}=\lambda _{2}.
\end{equation*}%
Then, from Theorem \ref{Cen 2} the $(m+1)\times (m+1)$ matrices 
\[
A+JB=\lambda _{1}\bigoplus_{j=1}^{\frac{m}{2}}%
\begin{bmatrix}
a_{1j} & -b_{1j} \\ 
b_{1j} & a_{1j}%
\end{bmatrix} \ \ \text{and} \ \ A-JB=\lambda _{2}\bigoplus_{j=1}^{\frac{m}{2}}
\begin{bmatrix}
a_{2j} & -b_{2j} \\ 
b_{2j} & a_{2j}%
\end{bmatrix} 
\]
have spectrum $\Lambda _{1}$  and $\Lambda_{2}$ respectively.  Then $A=\frac{1}{2}(A+JB+(A-JB))$ and $B=\frac{
1}{2}J(A+JB-(A-JB))$. Therefore 
\begin{equation*}
C=%
\begin{bmatrix}
A & JBJ \\ 
B & JAJ%
\end{bmatrix}%
\end{equation*}%
is real centrosymmetric with spectrum $\Lambda $. \newline
If $m$ is odd, we take the partition $\Lambda =\Lambda _{1}\cup \Lambda _{2}$
with 
\begin{eqnarray*}
\Lambda _{1} &=&\{\lambda _{11},\lambda _{12},z_{11},\overline{z}%
_{11},\ldots ,z_{1\lfloor \frac{m}{2}\rfloor },\overline{z}_{1\lfloor \frac{m%
}{2}\rfloor }\},\ \lambda _{11}=\lambda _{1},\ \lambda _{12}=\lambda _{2} \\
\Lambda _{2} &=&\{z_{21},\overline{z}_{21},\ldots ,z_{2\lceil \frac{m}{2}%
\rceil },\overline{z}_{2\lceil \frac{m}{2}\rceil }\}.
\end{eqnarray*}%
Then, 
\begin{equation*}
A+JB=\lambda _{1}\oplus \lambda _{2}\bigoplus_{j=1}^{\lfloor \frac{m}{2}%
\rfloor}
\begin{bmatrix}
a_{1j} & -b_{1j} \\ 
b_{1j} & a_{1j}
\end{bmatrix}
\text{ \ and}
\end{equation*}
\begin{equation*}
A-JB=\bigoplus_{j=1}^{\lceil \frac{m}{2}\rceil }%
\begin{bmatrix}
a_{2j} & -b_{2j} \\ 
b_{2j} & a_{2j}%
\end{bmatrix}%
\end{equation*}%
have the spectra $\Lambda _{1}$ and $\Lambda _{2},$ respectively and the
proof follows as above. It is clear that if $\Lambda $ has $r$ real numbers, 
$r,m$ even, we take the partition $\Lambda =\Lambda _{1}\cup \Lambda _{2}$
as above, with $\Lambda _{k},$ $k=1,2,$ having $m+\frac{r}{2}$ numbers, $m$
complex numbers and $\frac{r}{2}$ real numbers. If $m$ is odd and $\Lambda $
has $r$ real numbers, $r$ even, then $\Lambda _{1}$ will have $\lfloor \frac{%
m}{2}\rfloor $ complex numbers plus $\frac{r+2}{2}$ real numbers, while $%
\Lambda _{2}$ will have $\lceil \frac{m}{2}\rceil $ complex numbers plus $%
\frac{r-2}{2}$ real numbers. Then the proof follows as before.\\ \ \\
Case 2: Let $n$ be odd. First we consider $\Lambda $ with only one
real number, that is, 
\begin{equation*}
\Lambda =\{\lambda _{1},z_{1},\overline{z}_{1},\ldots ,z_{m},\overline{z}%
_{m}\},\ z_{j}=a_{j}+ib_{j},\ a_{j}\in \mathbb{R},\ b_{j}>0,\ j=1,\ldots
,m.
\end{equation*}%
If $m$ is even we take the partition $\Lambda =\Lambda _{1}\cup \Lambda _{2}$
with 
\begin{eqnarray*}
\Lambda _{1} &=&\{\lambda _{11},z_{11},\overline{z}_{11},\ldots ,z_{1\frac{m%
}{2}},\overline{z}_{1\frac{m}{2}}\},\ \lambda _{11}=\lambda _{1}, \\
\Lambda _{2} &=&\{z_{21},\overline{z}_{21},\ldots ,z_{2\frac{m}{2}},%
\overline{z}_{2\frac{m}{2}}\}.\ 
\end{eqnarray*}%
Then the matrices 
\begin{equation*}
\lambda _{1}\bigoplus_{j=1}^{\frac{m}{2}}%
\begin{bmatrix}
a_{1j} & -b_{1j} \\ 
b_{1j} & a_{1j}%
\end{bmatrix}%
=%
\begin{bmatrix}
\lambda _{1} &  \\ 
& A+JB
\end{bmatrix}%
\text{ \ and}
\end{equation*}%
\begin{equation*}
A-JB=\bigoplus_{j=1}^{\frac{m}{2}}%
\begin{bmatrix}
a_{2j} & -b_{2j} \\ 
b_{2j} & a_{2j}%
\end{bmatrix}%
,
\end{equation*}
have spectrum $\Lambda _{1}$ and $\Lambda _{2},$
respectively. Then $A=\frac{1}{2}(A+JB+(A-JB))$ and $B=\frac{
1}{2}J(A+JB-(A-JB))$. Hence from Theorem \ref{Cen 2}
\begin{equation*}
\begin{bmatrix}
A & \mathbf{0} & JBJ \\ 
\mathbf{0} & \lambda _{1} & \mathbf{0} \\ 
B & \mathbf{0} & JAJ%
\end{bmatrix}%
\end{equation*}%
is real centrosymmetric with spectrum $\Lambda $. \newline
On the other hand, if $m$ is odd, a partition $\Lambda =\Lambda _{1}\cup
\Lambda _{2}$ must be of the form 
\begin{eqnarray*}
\Lambda _{1} &=&\{z_{11},\overline{z}_{11},\ldots ,z_{1\lceil \frac{m}{2}%
\rceil },\overline{z}_{1\lceil \frac{m}{2}\rceil }\},\  \\
\Lambda _{2} &=&\{\lambda _{21},z_{21},\overline{z}_{21},\ldots ,z_{2\lfloor 
\frac{m}{2}\rfloor },\overline{z}_{2\lfloor \frac{m}{2}\rfloor }\},\text{ }%
\lambda _{21}=\lambda _{1},\ \ k=1,2
\end{eqnarray*}%
Then the matrices
\begin{equation*}
\bigoplus_{j=1}^{\lceil \frac{m}{2}\rceil }%
\begin{bmatrix}
a_{1j} & -b_{1j} \\ 
b_{1j} & a_{1j}%
\end{bmatrix}%
=%
\begin{bmatrix}
a_{11} & -b_{11} & 0 & 0  & \cdots  & 0 \\ 
b_{11} &  &  &  & & \\ 
0 &  &  & & &  \\ 0 &  &  & A+JB & &\\ 
\vdots  &  & & &  &  \\ 
0 &  &  &  & 
\end{bmatrix}%
\end{equation*}%
and
\begin{equation*}
A-JB=\lambda _{1}\bigoplus_{j=1}^{\lfloor \frac{m}{2}\rfloor }%
\begin{bmatrix}
a_{2j} & -b_{2j} \\ 
b_{2j} & a_{2j}
\end{bmatrix}
\end{equation*}
have spectrum $\Lambda _{1}$ and $\Lambda _{2}$ respectively. Then $A=\frac{1}{2}(A+JB+(A-JB))$ and $B=\frac{
1}{2}J(A+JB-(A-JB))$. 
Hence  
\begin{equation*}
\begin{bmatrix}
A & \mathbf{x} & JBJ \\ 
\mathbf{\mathbf{y}}^{T} & a_{11} & \mathbf{y}^{T}J \\ 
B & J\mathbf{x} & JAJ%
\end{bmatrix}%
\end{equation*}%
where $\mathbf{x}=\frac{1}{\sqrt{2}}[b_{11},0\cdots ,0]^{T}$ and $\mathbf{y}%
^{T}=\frac{1}{\sqrt{2}}[-b_{11},\cdots ,0]$, is a real centrosymmetric matrix
with spectrum $\Lambda $.\newline
It is easy to see that the list $\Lambda =\{\lambda _{1},z_{1},\overline{z}%
_{1},\ldots ,z_{m},\overline{z}_{m}\}$ can be extended to a list $\Lambda
^{\prime }=\{\lambda _{1},\ldots ,\lambda _{p},z_{1},\overline{z}_{1},\ldots
,z_{m},\overline{z}_{m}\}$ with $p$ real numbers, $p$ odd, which always
admit a partition in two self-conjugate lists $\Lambda _{1}$, $\Lambda _{2}$
such that $\Lambda _{1}$ has $\lceil \frac{2m+p}{2}\rceil $ eigenvalues and $%
\Lambda _{2}$ has the remaining eigenvalues, and a real centrosymmetric
matrix with spectrum $\Lambda ^{\prime }$ can be obtained as above.
\end{proof}

\section{Centrosymmetric nonnegative inverse eigenvalue problem}
In this section we study the NIEP for centrosymmetric matrices. First, we
show that lists of real nonnegative numbers are always realizable for a centrosymmetric nonnegative matrix. Second, we show that lists of complex
numbers of Suleimanova type \cite{Sule} are also realizable for centrosymmetric nonnegative
matrices, except if the list has only one real number and $m$
pairs of conjugated complex numbers, with $m$ being odd. Third, for the general lists, we give  sufficient conditions for the existence of a centrosymmetric nonnegative matrix with prescribed complex spectrum via a perturbation result. Finally,
we study the centrosymmetric realizability of lists of complex numbers of
size $n=4$ with prescribed diagonal entries. \\ \\
We start by showing that a list of real nonnegative numbers $\Lambda
=\{\lambda _{1},\ldots ,\lambda _{n}\}$ is always realizable by a
centrosymmetric matrix. In addition, if $\lambda _{1}$ is simple, $\Lambda $
is realizable by a centrosymmetric positive matrix.
\begin{theorem}
\label{th. real} Let $\Lambda =\{\lambda _{1},\lambda _{2},\ldots ,\lambda
_{n}\}$ be a list of nonnegative real numbers with $\lambda _{1}\geq\lambda
_{2}\geq \cdots \geq \lambda _{n}\geq 0$. Then $\Lambda $ is realizable by
an $n\times n$ centrosymmetric matrix.
\end{theorem}
\begin{proof}
For even $n=2m,$ we take the partition $\Lambda =\Lambda _{1}\cup \Lambda
_{2} $, where 
\begin{equation*}
\Lambda _{1}=\{\lambda _{1},\ldots ,\lambda _{m}\},\ \ \Lambda
_{2}=\{\lambda _{m+1},\ldots ,\lambda _{n}\}.
\end{equation*}%
Then for $A+JB=diag\{\lambda _{1},\ldots ,\lambda _{m}\}$ and $%
A-JB=diag\{\lambda _{m+1},\ldots ,\lambda _{n}\}$ we have that 
\begin{equation*}
A=\frac{1}{2}diag\{\lambda _{1}+\lambda _{m+1},\ldots ,\lambda _{m}+\lambda
_{n}\},\ \ B=\frac{1}{2}Jdiag\{\lambda _{1}-\lambda _{m+1},\ldots ,\lambda
_{m}-\lambda _{n}\},
\end{equation*}%
are nonnegative matrices, and a solution matrix is of the form 
\begin{equation*}
C=%
\begin{bmatrix}
A & JBJ \\ 
B & JAJ%
\end{bmatrix}%
.
\end{equation*}%
For odd $n=2m+1,$ we take the partition $\Lambda =\Lambda _{1}\cup \Lambda
_{2}$, where 
\begin{equation*}
\Lambda _{1}=\{\lambda _{1},\ldots ,\lambda _{\lceil \frac{n}{2}\rceil }\},\
\ \Lambda _{2}=\{\lambda _{\lceil \frac{n}{2}\rceil +1},\ldots ,\lambda
_{n}\}.
\end{equation*}%
Then for 
\begin{equation*}
\begin{bmatrix}
A+JB & \sqrt{2}\mathbf{x} \\ 
\sqrt{2}\mathbf{y}^{T} & c%
\end{bmatrix}%
=diag\{\lambda _{1},\ldots ,\lambda _{\lceil \frac{n}{2}\rceil }\}
\end{equation*}%
and $A-JB=diag\{\lambda _{\lceil \frac{n}{2}\rceil +1},\ldots ,\lambda
_{n}\},$ we have $A+JB=diag\{\lambda _{1},\ldots ,\lambda _{\lfloor \frac{n}{%
2}\rfloor }\}$, $c=\lambda _{\lceil \frac{n}{2}\rceil }$. Therefore%
\begin{eqnarray*}
A &=&\frac{1}{2}diag\{\lambda _{1}+\lambda _{\lceil \frac{n}{2}\rceil
+1},\ldots ,\lambda _{\lfloor \frac{n}{2}\rfloor }+\lambda _{n}\}, \\
B &=&\frac{1}{2}Jdiag\{\lambda _{1}-\lambda _{\lceil \frac{n}{2}\rceil
+1},\ldots ,\lambda _{\lfloor \frac{n}{2}\rfloor }-\lambda _{n}\},
\end{eqnarray*}%
are nonnegative matrices and a solution matrix is 
\begin{equation*}
C=%
\begin{bmatrix}
A & \mathbf{0} & JBJ \\ 
\mathbf{0} & \lambda _{\lceil \frac{n}{2}\rceil } & \mathbf{0} \\ 
B & \mathbf{0} & JAJ%
\end{bmatrix}%
.
\end{equation*}
\end{proof}\ \\
In \cite{Perfect} Perfect introduces the $n\times n$ matrix 
\begin{equation}
P=%
\begin{bmatrix}
1 & 1 & \cdots & 1 & 1 \\ 
1 & 1 & \cdots & 1 & -1 \\ 
1 & 1 & \cdots & -1 & 0 \\ 
\vdots & \vdots & \cdots & \vdots & \vdots \\ 
1 & -1 & \cdots & 0 & 0%
\end{bmatrix}%
,  \label{Per1}
\end{equation}%
and she proves that if $D=diag\{\lambda _{1},\lambda _{2},\ldots ,\lambda
_{n}\},$ with $\lambda _{1}>\lambda _{2}\geq \cdots \geq \lambda _{n}\geq 0,$
then $A=PDP^{-1}$ is a positive matrix in $\mathcal{CS}_{\lambda _{1}}.$ As
a consequence, we have the following result, which gives a very simple way
to compute a centrosymmetric positive matrix with prescribed nonnegative spectrum.
\begin{theorem}
Let $\Lambda =\{\lambda _{1},\lambda _{2},\ldots ,\lambda _{n}\}$ be a list
of nonnegative real numbers with $\lambda _{1}>\lambda _{2}\geq \cdots \geq
\lambda _{n}\geq 0$. Then $\Lambda $ is realizable by an $n\times n$
centrosymmetric positive matrix.
\end{theorem}
\begin{proof}
For even $n=2m,$ we take the partition \ $\Lambda =\Lambda _{1}\cup \Lambda
_{2}$, with 
\begin{equation*}
\Lambda _{1}=\{\lambda _{1},\ldots ,\lambda _{m}\},\ \ \Lambda
_{2}=\{\lambda _{m+1},\ldots ,\lambda _{n}\}.
\end{equation*}%
We set $A+JB=PDP^{-1}>0$, where $D=diag\{\lambda _{1},\ldots ,\lambda
_{m}\}, $ $P$ is the Perfect matrix in \eqref{Per1}, and $A-JB=diag\{\lambda
_{m+1},\ldots ,\lambda _{n}\}$. Then 
\begin{equation*}
A=\frac{1}{2}(PDP^{-1}+diag\{\lambda _{m+1},\ldots ,\lambda _{n}\})>0,
\end{equation*}%
and 
\begin{equation*}
B=\frac{1}{2}J(PDP^{-1}-diag\{\lambda _{m+1},\ldots ,\lambda _{n}\}).
\end{equation*}%
Therefore, 
\begin{equation}
C=%
\begin{bmatrix}
A & JBJ \\ 
B & JAJ%
\end{bmatrix}
\label{c}
\end{equation}%
is a centrosymmetric matrix with spectrum $\Lambda $. It remains to prove
that $C$ is positive. It is enough to show that $B$ is positive. In fact, if 
$d_{jj}$, $j=1,2,\ldots ,m,$ are the diagonal entries of $PDP^{-1},$ we must
show that 
\begin{equation*}
d_{jj}>\lambda _{m+j},\ \ \ \text{for all}\ \ j=1,2,\ldots ,m.
\end{equation*}%
It is easy to see that the diagonal entries of the matrix $PDP^{-1}$ are
given by: 
\begin{align*}
d_{11}& =\frac{1}{2^{m-1}}\lambda _{1}+\frac{1}{2^{m-1}}\lambda _{2}+\frac{1%
}{2^{m-2}}\lambda _{3}+\cdots +\frac{1}{2^{2}}\lambda _{m-1}+\frac{1}{2}%
\lambda _{m}=d_{22} \\
d_{33}& =\frac{1}{2^{m-2}}\lambda _{1}+\frac{1}{2^{m-2}}\lambda _{2}+\frac{1%
}{2^{m-3}}\lambda _{3}+\cdots +\frac{1}{2^{2}}\lambda _{m-2}+\frac{1}{2}%
\lambda _{m-1} \\
& \vdots \\
& \vdots \\
d_{m-1,m-1}& =\frac{1}{2^{2}}\lambda _{1}+\frac{1}{2^{2}}\lambda _{2}+\frac{1%
}{2}\lambda _{3} \\
d_{mm}& =\frac{1}{2}\lambda _{1}+\frac{1}{2}\lambda _{2}.
\end{align*}%
For $j=1,$ $d_{11}>\lambda _{m+1}.$ In fact: since%
\begin{equation*}
\frac{1}{2^{m-1}}\lambda _{1}>\frac{1}{2^{m-1}}\lambda _{m+1},\ \frac{1}{%
2^{m-1}}\lambda _{2}\geq \frac{1}{2^{m-1}}\lambda _{m+1},\ \ldots ,\ \frac{1%
}{2}\lambda _{m}\geq \frac{1}{2}\lambda _{m+1},
\end{equation*}%
then by adding the inequalities we have%
\begin{align*}
d_{11}& >(\frac{1}{2^{m-1}}+\frac{1}{2^{m-1}}+\frac{1}{2^{m-2}}+\cdots +%
\frac{1}{2^{2}}+\frac{1}{2})\lambda _{m+1} \\
& =\frac{1+1+2+2^{2}+\cdots +2^{m-3}+2^{m-2}}{2^{m-1}}\lambda _{m+1} \\
& =\frac{1+(2^{m-1}-1)}{2^{m-1}}\lambda _{m+1} \\
& =\lambda _{m+1}.
\end{align*}%
Since $d_{11}=d_{22}$ and $\lambda _{m+1}\geq \lambda _{m+2}$, then $%
d_{22}>\lambda _{m+2}$.\newline
By proceeding in the same way we have 
$d_{jj}>\lambda _{m+j},$ $j=3,4,\ldots ,m$.
Therefore $B>0$, and the matrix $C$ in \eqref{c} is centrosymmetric positive. \newline
If $n=2m+1$ is odd, then we take the partition $\Lambda =\Lambda _{1}\cup
\Lambda _{2}$ with%
\begin{equation*}
\Lambda _{1}=\{\lambda _{1},\ldots ,\lambda _{\lceil \frac{n}{2}\rceil }\},\
\ \ \Lambda _{2}=\{\lambda _{\lceil \frac{n}{2}\rceil +1},\ldots ,\lambda
_{n}\}.
\end{equation*}%
Then for $%
\begin{bmatrix}
A+JB & \sqrt{2}\mathbf{x} \\ 
\sqrt{2}\mathbf{y}^{T} & c%
\end{bmatrix}%
=PDP^{-1}=%
\begin{bmatrix}
A_{11} & \mathbf{a} \\ 
\mathbf{b}^{T} & a_{m+1,m+1}%
\end{bmatrix}%
>0$, and\newline
$A-JB=diag\{\lambda _{\lceil \frac{n}{2}\rceil +1},\ldots ,\lambda _{n}\},$
we have 
\begin{equation*}
A=\frac{1}{2}(A_{11}+diag\{\lambda _{\lceil \frac{n}{2}\rceil +1},\ldots
,\lambda _{n}\})>0,
\end{equation*}%
\begin{equation*}
B=\frac{1}{2}J(A_{11}-diag\{\lambda _{\lceil \frac{n}{2}\rceil +1},\ldots
,\lambda _{n}\}).
\end{equation*}%
$B$ is positive from the same argument as in the even case. Therefore the
matrix 
\begin{equation*}
C=%
\begin{bmatrix}
A & \frac{1}{\sqrt{2}}\mathbf{a} & JBJ \\ 
\frac{1}{%
\sqrt{2}}\mathbf{b}^{T} & a_{m+1,m+1} & \frac{1}{%
\sqrt{2}}\mathbf{b}^{T}J \\ 
B & J\frac{1}{\sqrt{2}}\mathbf{a} & JAJ%
\end{bmatrix}%
,
\end{equation*}
is centrosymmetric positive with spectrum $\Lambda
.$
\end{proof}\ \\ \\
Next we shall see an anomalous case. That is, we shall prove that if a list
realizable $\Lambda $ has only one real positive number and $m$ pairs of
conjugated complex numbers, with $m$ being odd, then $\Lambda $ cannot be
the spectrum of a centrosymmetric nonnegative matrix. 
\begin{theorem}
\label{NO realizable} Let $\Lambda =\{\lambda _{1},z_{1},\overline{z}%
_{1},\ldots ,z_{m},\overline{z}_{m}\}$ be a realizable list of complex
numbers, with odd $m$, $z_{j}=a_{j}+ib_{j}$ $j=1,\ldots,m$, $a\in \mathbb{R}$, $b_{j}>0$. Then $\Lambda $ cannot be the spectrum of a centrosymmetric nonnegative matrix.
\end{theorem}
\begin{proof}
Suppose that $\Lambda $ is the spectrum of a centrosymmetric nonnegative matrix $C$ of order $2m+1.$ Then $C$ is of the form 
\begin{equation*}
C=
\begin{bmatrix}
A & \mathbf{x} & JBJ \\ 
\mathbf{y}^{T} & c & \mathbf{y}^{T}J \\ 
B & J\mathbf{x} & JAJ%
\end{bmatrix},
\end{equation*}
where $A$ and $B$ are $m\times m$ nonnegative matrices, $\mathbf{x}$, $%
\mathbf{y}$ are nonnegative $m\times 1$ matrices, and $c$ is a nonnegative
real number. From Theorem \ref{Cen 2} $C$ is orthogonally similar to the
matrix 
\begin{equation*}
\begin{bmatrix}
A-JB &  &  \\ 
& c & \sqrt{2}\mathbf{y}^{T} \\ 
& \sqrt{2}\mathbf{x} & A+JB
\end{bmatrix},
\end{equation*}
where $
\begin{bmatrix}
c & \sqrt{2}\mathbf{y}^{T} \\ 
\sqrt{2}\mathbf{x} & A+JB%
\end{bmatrix}%
$ is an $(m+1)\times (m+1)$ nonnegative matrix with $m+1$ eigenvalues of $C$
including the Perron eigenvalue. That is, $C$ has $m$ (an odd number)
complex eigenvalues, which is a contradiction.
\end{proof}\ \\ \ \\
Now we consider the centrosymmetric realizability of lists of Suleimanova
type. We start with the following simple result:
\begin{lemma}
\label{cent} Let $\Lambda =\{\lambda _{1},\lambda _{2},\ldots ,\lambda_{n}\} $ be the spectrum of a centrosymmetric nonnegative matrix and let $\epsilon >0$. Then $\{\lambda _{1}+\epsilon ,\lambda _{2},\ldots ,\lambda
_{n}\}$ is also the spectrum of a centrosymmetric nonnegative matrix.
\end{lemma}\ 
\begin{proof}
Let $C$ be a centrosymmetric nonnegative matrix with spectrum $\Lambda $.
Then from Lemma \ref{eigen}, there exists $\mathbf{v}\geq 0$, $C\mathbf{v}%
=\lambda _{1}\mathbf{v}$, with $J\mathbf{v}=\mathbf{v,}$ and from Brauer's
Theorem (see \eqref{Bra}) the matrix 
\begin{equation*}
C+\frac{\epsilon }{\mathbf{v}^{T}\mathbf{v}}\mathbf{v}\mathbf{v}^{T},
\end{equation*}%
is centrosymmetric nonnegative with spectrum $\{\lambda _{1}+\epsilon
,\lambda _{2},\ldots ,\lambda _{n}\}$.
\end{proof}\ \\ \\
It is well known that lists of complex numbers $\Lambda =\{\lambda
_{1},\lambda _{2},\ldots ,\lambda _{n}\}$ of Suleimanova type are realizable
by a nonnegative matrix if only if $\sum_{i=1}^{n}\lambda _{i}\geq 0$. We
prove that a list of this type is in particular realizable by a centrosymmetric nonnegative
matrix, except if the list is as in the Theorem \ref{NO
realizable}, that is, it has only one real positive number and $m$ pairs of
conjugated complex numbers, with $m$ being odd. \\ \\
The following result show that any realizable list of complex numbers of Suleimanova
type, with two real numbers and $m$ pairs of complex numbers is always the
spectrum of a centrosymmetric nonnegative matrix of order $n=2m+2.$ As a
consequence, all realizable lists of complex numbers of Suleimanova type, with an even
number of elements, are always realizable by centrosymmetric matrices.
\begin{lemma}
\label{2real}Let $\Lambda =\{\lambda _{1},\lambda _{2},z_{1},\overline{z}%
_{1},\ldots ,z_{m},\overline{z}_{m}\}$, with $\lambda _{2},z_{j},\overline{z}%
_{j}\in \mathcal{F}$, $z_{j}=a_{j}+ib_{j}$, $a_{j}\in \mathbb{R}$, $b_{j}>0$, $j=1,\ldots ,m,$ be a realizable list of complex
numbers. Then $\Lambda $ is the spectrum of a  centrosymmetric nonnegative
matrix.
\end{lemma}
\begin{proof}
Since $\Lambda $ is realizable if only if $\sum_{i=1}^{n}\lambda _{i}\geq 0$, we take the list 
\begin{equation*}
\Lambda ^{\prime }=\{\lambda _{1}^{\prime },\lambda _{2},z_{1},\overline{z}%
_{1},\ldots ,z_{m},\overline{z}_{m}\},\ \text{with }\lambda _{1}^{\prime
}=-\lambda _{2}-2\sum_{j=1}^{m}a_{j},\ \text{and}\ a_{j}=Rez_{j}.
\end{equation*}%
For even $m$ we take the partition $\Lambda =\Lambda _{1}\cup \Lambda _{2}$
with 
\begin{eqnarray*}
\Lambda _{1} &=&\{\lambda _{11},z_{11},\overline{z}_{11},\ldots ,z_{1\frac{m%
}{2}},\overline{z}_{1\frac{m}{2}}\},\ \lambda _{11}=\lambda _{1}^{\prime },
\\
\Lambda _{2} &=&\{\lambda _{21},z_{21},\overline{z}_{21},\ldots ,z_{2\frac{m%
}{2}},\overline{z}_{2\frac{m}{2}}\},\ \lambda _{21}=\lambda _{2}.
\end{eqnarray*}%
Then,
\begin{equation*}
A-JB=\lambda _{2}\bigoplus_{j=1}^{\frac{m}{2}}%
\begin{bmatrix}
a_{2j} & -b_{2j} \\ 
b_{2j} & a_{2j}%
\end{bmatrix}%
\end{equation*}%
has spectrum $\Lambda _{2}$, and from Lemma \ref{SJC} we can always compute a
nonnegative matrix $A+JB$ with spectrum $\Lambda _{1}$ and diagonal entries 
\begin{equation*}
-\lambda _{2},-a_{21},-a_{21},\ldots ,-a_{2\frac{m}{2}},-a_{2\frac{m}{2}}.
\end{equation*} 
Since $\left\vert Rez _{j}\right\vert \geq \left\vert 
Imz _{j}\right\vert ,$ $j=1,\ldots ,m,$ both matrices,%
\begin{equation*}
A=\frac{1}{2}(A+JB+A-JB)\ \ \text{and}\ \ B=\frac{1}{2}J((A+JB)-(A-JB)).
\end{equation*}%
are nonnegative and 
\begin{equation*}
C^{\prime }=%
\begin{bmatrix}
A & JBJ \\ 
B & JAJ%
\end{bmatrix}%
\end{equation*}%
is centrosymmetric nonnegative with spectrum $\Lambda ^{\prime }$. Finally,
from Lemma \ref{cent}, $C=C^{\prime }+\frac{\sum_{i=1}^{n}\lambda _{i}}{%
\mathbf{v}^{T}\mathbf{v}}\mathbf{v}\mathbf{v}^{T}$, where $\mathbf{v}$ is
the Perron eigenvector of $C^{\prime }$, is centrosymmetric nonnegative with spectrum $\Lambda $. \\ \\
For odd $m$ we take the partition $\Lambda =\Lambda _{1}\cup \Lambda _{2}$
with 
\begin{eqnarray*}
\Lambda _{1} &=&\{\lambda _{11},\lambda _{12},z_{11},\overline{z}%
_{11},\ldots ,z_{1\lfloor \frac{m}{2}\rfloor },\overline{z}_{1\lfloor \frac{m%
}{2}\rfloor }\},\  \\
\Lambda _{2} &=&\{z_{21},\overline{z}_{21},\ldots ,z_{2\lceil \frac{m}{2}%
\rceil },\overline{z}_{2\lceil \frac{m}{2}\rceil }\},\text{ }\lambda
_{11}=\lambda _{1}^{\prime },\ \ \lambda _{12}=\lambda _{2},
\end{eqnarray*}%
and the proof follows as above.
\end{proof}\ \\
More generally, if $\Lambda $ has $r\geq 2$ real numbers, with even $r$, we
have the following result:

\begin{corollary}
\label{coro1} Let $\Lambda =\{\lambda _{1},\ldots,\lambda_{r},z_{1},\overline{z}_{1},\ldots ,z_{m},\overline{z}_{m}\}$, with $z_{j}, 
\overline{z}_{j}\in \mathcal{F}$, $z_{j}=a_{j}+ib_{j}$, $a_{j}\in \mathbb{R}$, $b_{j}>0$, $j=1,\ldots,m$; $\lambda_{j}\in \mathcal{F}$, $j=2,\ldots, r$, be a realizable list of complex numbers. Then $\Lambda$ is the spectrum of a centrosymmetric nonnegative matrix.
\end{corollary}
\begin{proof}
We consider the list 
\begin{eqnarray*}
\Lambda ^{\prime } &=&\{\lambda _{1}^{\prime },\lambda _{2},\ldots ,\lambda
_{r},z_{1},\overline{z}_{1},\ldots ,z_{m},\overline{z}_{m}\},\text{with}\ \ 
\\
\lambda _{1}^{\prime } &=&-\sum_{i=2}^{r}\lambda _{i}-2\sum_{j=1}^{m}a_{j},%
\text{ where }a_{j}=Rez_{j},\text{ }j=1,\ldots ,m.
\end{eqnarray*}%
Then, for even $m$ we take the partition of $\Lambda ^{\prime }=\Lambda
_{1}\cup \Lambda _{2}$ with: 
\begin{eqnarray*}
\Lambda _{1} &=&\{\lambda _{11},\lambda _{12},\ldots ,\lambda _{1\frac{r}{2}%
},z_{11},\overline{z}_{11},\ldots ,z_{1\frac{m}{2}},\overline{z}_{1\frac{m}{2%
}}\},\ \lambda _{11}=\lambda _{1}^{\prime },\  \\
\Lambda _{2} &=&\{\lambda _{21},\lambda _{22},\ldots ,\lambda _{2\frac{r}{2}%
},z_{21},\overline{z}_{21},\ldots ,z_{2\frac{m}{2}},\overline{z}_{2\frac{m}{2%
}}\},\ 
\end{eqnarray*}%
while for odd $m$ we take 
\begin{eqnarray*}
\Lambda _{1} &=&\{\lambda _{11},\lambda _{12},\ldots ,\lambda _{1,\frac{r}{2}%
+1},z_{11},\overline{z}_{11},\ldots ,z_{1\lfloor \frac{m}{2}\rfloor },%
\overline{z}_{1\lfloor \frac{m}{2}\rfloor }\},\ \lambda _{11}=\lambda
_{1}^{\prime } \\
\Lambda _{2} &=&\{\lambda _{21},\lambda _{22},\ldots ,\lambda _{2,\frac{r}{2}%
-1},z_{21},\overline{z}_{21},\ldots ,z_{2\lceil \frac{m}{2}\rceil },%
\overline{z}_{2\lceil \frac{m}{2}\rceil }\},
\end{eqnarray*}%
and the proof follows as the proof of Lemma \ref{2real}.
\end{proof}\ \\ \ \\
Now, we consider lists of complex numbers of Suleimanova type with an odd
number of elements. We start with the following result:
\begin{lemma}
\label{Lema} Let $\Lambda =\{\lambda _{1},z_{1},\overline{z}_{1},\ldots
,z_{m},\overline{z}_{m}\}$ be a list of complex numbers with $z_{j},%
\overline{z}_{j}\in \mathcal{F}$, $z_{j}=a_{j}+ib_{j}$, $a_{j}\in\mathbb{R}$, $b_{j}>0$, $j=1,\ldots ,m,$ $m$ even, which is
realizable. Then $\Lambda $ is the spectrum of a centrosymmetric nonnegative matrix.
\end{lemma}

\begin{proof}
We consider the list 
\begin{equation*}
\Lambda ^{\prime }=\{-2\sum_{j=1}^{m}a_{j},z_{1},\overline{z}_{1},\ldots
,z_{m},\overline{z}_{m}\},
\end{equation*}%
with the partition $\Lambda ^{\prime }=\Lambda _{1}\cup \Lambda _{2}$, where 
\begin{eqnarray*}
\Lambda _{1} &=&\{-2\sum_{j=1}^{m}a_{j},z_{11},\overline{z}_{11},\ldots ,z_{1
\frac{m}{2}},\overline{z}_{1\frac{m}{2}}\} \\
\Lambda _{2} &=&\{z_{21},\overline{z}_{21},\ldots
,z_{2\frac{m}{2}},\overline{z}_{2\frac{m}{2}}\}.
\end{eqnarray*}%
Then 
\begin{equation*}
A-JB=\bigoplus _{j=1}^{\frac{m}{2}}
\begin{bmatrix}
a_{2j} & -b_{2j} \\ 
b_{2j} & a_{2j}
\end{bmatrix}
\end{equation*}%
is an $m\times m$ real matrix with spectrum $\Lambda _{2},$ and 
\begin{equation*}
\begin{bmatrix}
c & \sqrt{2}\mathbf{y}^{T} \\ 
\sqrt{2}\mathbf{x} & A+JB%
\end{bmatrix}%
=%
\begin{bmatrix}
0 & \mathbf{a}^{T} \\ 
\mathbf{b} & A_{22}%
\end{bmatrix}%
\geq 0,
\end{equation*}%
of order $m+1,$ has spectrum $\Lambda _{1}$ and diagonal entries 
\begin{equation*}
0,-a_{21},-a_{21},\ldots ,-a_{2\frac{m}{2}},-a_{2\frac{m}{2}},
\end{equation*}%
which there exists from Lemma \ref{SJC}. Then $A=\frac{1}{2}(A_{22}+A-JB)$
and $B=\frac{1}{2}J(A_{22}-(A-JB))$, are both nonnegative and the
matrix 
\begin{equation*}
C^{\prime }=%
\begin{bmatrix}
A & \frac{\mathbf{b}}{\sqrt{2}} & JBJ \\ 
\frac{\mathbf{a}^{T}}{\sqrt{2}} & 0 & \frac{\mathbf{a}^{T}}{\sqrt{2}}J \\ 
B & J\frac{\mathbf{b}}{\sqrt{2}} & JAJ%
\end{bmatrix}%
\end{equation*}%
is centrosymmetric nonnegative with spectrum $\Lambda ^{\prime }$. Thus, $%
C=C^{\prime }+\frac{\sum_{i=1}^{n}\lambda _{i}}{\mathbf{v}^{T}\mathbf{v}}%
\mathbf{v}\mathbf{v}^{T}$, where $\mathbf{v}$ is the Perron eigenvector of $%
C^{\prime }$, is centrosymmetric nonnegative with spectrum $\Lambda $.%
\end{proof}\ \\ \\
More generally, if $\Lambda $ has $p$ real numbers, with odd $p,$ we have
the following result:
\begin{corollary}
Let $\Lambda =\{\lambda _{1},\ldots ,\lambda _{p},z_{1},\overline{z}%
_{1},\ldots ,z_{m},\overline{z}_{m}\}$, with $z_{j},\overline{z}_{j}\in 
\mathcal{F}$, $z_{j}=a_{j}+ib_{j}$, $a_{j}\in\mathbb{R}$, $b_{j}>0$, $j=1,\ldots ,m$; $\lambda _{k}\in \mathcal{F},$ $k=2,\ldots
,p, $ be a realizable list of complex numbers. Then $\Lambda $ is the
spectrum of a centrosymmetric nonnegative matrix. 
\end{corollary}
\begin{proof}
Case even $m$.
We consider the list 
\begin{equation*}
\Lambda ^{\prime }=\{-\sum_{k=2}^{p}\lambda
_{k}-2\sum_{j=1}^{m}a_{j},\lambda _{2},\ldots ,\lambda _{p},z_{1},\overline{z%
}_{1},\ldots ,z_{m},\overline{z}_{m}\},
\end{equation*}%
and take the partition $\Lambda ^{\prime }=\Lambda _{1}\cup \Lambda _{2}$
with, 
\begin{eqnarray*}
\Lambda _{1} &=&\{-\sum_{k=2}^{p}\lambda _{k}-2\sum_{j=1}^{n}a_{j},\lambda
_{2},\ldots ,\lambda _{\lceil \frac{p}{2}\rceil },z_{11},\overline{z}%
_{11},\ldots ,z_{1\frac{m}{2}},\overline{z}_{1\frac{m}{2}}\} \\
&& \\
\Lambda _{2} &=&\{\lambda _{\lceil \frac{p}{2}\rceil +1},\ldots ,\lambda
_{p},z_{21},\overline{z}_{21},\ldots ,z_{2\frac{m}{2}},\overline{z
}_{2\frac{m}{2}}\}.
\end{eqnarray*}
Then 
\begin{equation*}
A-JB=diag\{\lambda _{\lceil \frac{p}{2}\rceil +1},\ldots ,\lambda
_{p}\}\bigoplus _{j=1}^{\frac{m}{2}}%
\begin{bmatrix}
a_{2j} & -b_{2j} \\ 
b_{2j} & a_{2j}
\end{bmatrix}%
\end{equation*}%
is a real matrix with spectrum $\Lambda _{2}$, and there exists a
nonnegative matrix $A+JB$ with spectrum $\Lambda _{1}$ and diagonal entries 
\begin{equation*}
0,-\lambda _{\lceil \frac{p}{2}\rceil +1},\ldots ,-\lambda _{p},-a_{21},-a_{21},\ldots ,-a_{2\frac{m}{2}},-a_{2\frac{m}{2}}.
\end{equation*}%
Then the proof follows as before. \newline
\ \newline
Case odd $m$: We consider a partition $\Lambda =\Lambda _{1}\cup \Lambda _{2}$, in which
the list $\Lambda _{1}$ has the Perron eigenvalue and only one more element
than the list $\Lambda _{2}.$ Both lists, $\Lambda _{1}$ and $\Lambda _{2}$
must be self-conjugated. Partitions of this type always exist if $p\geq 3.$
The construction of a nonnegative centrosymmetric matrix follows as in the
proof of Lemma \ref{Lema}.
\end{proof}\ \\
\begin{remark}
\label{R2} According to the results in this section, it is clear that a
realizable list of real numbers of Suleimanova type is\ in particular
realizable by a centrosymmetric nonnegative  matrix.
\end{remark}\ \\
We conclude this section by presenting a sufficient condition for the
existence and construction of centrosymmetric nonnegative  matrices with
prescribed spectrum\textbf{.} \newline
Let $\Lambda =\{\lambda _{1},\lambda _{2},\ldots ,\lambda _{n}\}$ be a
realizable list of complex numbers. In order to construct a centrosymmetric nonnegative matrix $A$ with spectrum $\Lambda $ we consider the
following partition of $\Lambda :$%
\begin{equation}
\left. 
\begin{array}{c}
\Lambda =\Lambda _{0}\cup \Lambda _{1}\cup \cdots \cup \Lambda _{\frac{p_{0}%
}{2}}\cup \Lambda _{\frac{p_{0}}{2}}\cup \cdots \cup \Lambda _{1},\text{ for
even }p_{0}, \\ 
\\ 
\Lambda =\Lambda _{0}\cup \Lambda _{1}\cup \cdots \cup \Lambda _{\lfloor 
\frac{p_{0}}{2}\rfloor }\cup \Lambda _{\lceil \frac{p_{0}}{2}\rceil }\cup
\Lambda _{\lfloor \frac{p_{0}}{2}\rfloor }\cup \cdots \cup \Lambda _{1},%
\text{ for odd }p_{0}, \\ 
\\ 
\text{with \ \ \ \ \ \ \ }\Lambda _{0}=\{\lambda _{01},\lambda _{02},\ldots
,\lambda _{0p_{0}}\},\text{ \ }\lambda _{01}=\lambda _{1}, \\ 
\\ 
\Lambda _{k}=\{\lambda _{k1},\lambda _{k2},\ldots ,\lambda _{kp_{k}}\},\text{
\ }k=1,2,\ldots ,\frac{p_{0}}{2}\text{ }(\lceil \frac{p_{0}}{2}\rceil \text{
for odd }p_{0}),%
\end{array}%
\right\}  \label{parti}
\end{equation}%
where some of the lists $\Lambda _{k}$ can be empty. For each list $\Lambda
_{k},$ $k\neq \lceil \frac{p_{0}}{2}\rceil ,$ we associate the list%
\begin{equation}
\Gamma _{k}=\{\omega _{k},\lambda _{k1},\lambda _{k2},\ldots ,\lambda
_{kp_{k}}\},\text{ }0\leq \omega _{k}\leq \lambda _{1},\text{\ }k=1,2,\ldots
,\frac{p_{0}}{2}(\lfloor \frac{p_{0}}{2}\rfloor \text{ for odd }p_{0}),
\label{gama}
\end{equation}%
which is realizable by a $(p_{k}+1)\times (p_{k}+1)$ nonnegative matrix $%
A_{k}.$ In particular, $A_{k}$ can be chosen as $A_{k}\in \mathcal{CS}%
_{\omega _{k}}.$ To the sub-list $\Lambda _{\lceil \frac{p_{0}}{2}\rceil }$
we associate the list $\Gamma _{\lceil \frac{p_{0}}{2}\rceil },$ which is
the spectrum of a centrosymmetric nonnegative matrix $A_{\lceil \frac{p_{0}}{%
2}\rceil }$ of order $p_{\lceil \frac{p_{0}}{2}\rceil +1}.$ Then, the block
diagonal matrices%
\begin{equation}
\left. 
\begin{array}{c}
A=diag\{A_{1},A_{2},\ldots ,A_{\frac{p_{0}}{2}},JA_{\frac{p_{0}}{2}}J,\ldots
,JA_{2}J,JA_{1}J\},\text{\ for even }p_{0},\text{ and} \\ 
A=diag\{A_{1},A_{2},\ldots ,A_{\lfloor \frac{p_{0}}{2}\rfloor },A_{\lceil 
\frac{p_{0}}{2}\rceil },JA_{\lfloor \frac{p_{0}}{2}\rfloor }J,\ldots
,JA_{2}J,JA_{1}J\},\text{ for odd }p_{0},%
\end{array}%
\right\}  \label{Apis}
\end{equation}%
are centrosymmetric nonnegative with spectrum 
\begin{eqnarray*}
\Gamma _{1}\cup \Gamma _{2}\cup \cdots \cup \Gamma _{\frac{%
p_{0}}{2}}\cup \Gamma _{\frac{p_{0}}{2}}\cup \cdots \cup \Gamma _{2}\cup
\Gamma _{1},\text{ even }p_{0}\text{\ and} \\
\Gamma _{1}\cup \Gamma _{2}\cup \cdots \cup \Gamma _{\lfloor 
\frac{p_{0}}{2}\rfloor }\cup \Gamma _{\lceil \frac{p_{0}}{2}\rceil }\cup
\Gamma _{\lfloor \frac{p_{0}}{2}\rfloor }\cup \cdots \cup \Gamma _{2}\cup
\Gamma _{1},\text{ odd }p_{0},
\end{eqnarray*}%
respectively. Of course, the matrices $J$ in \eqref{Apis} have the same order $%
p_{k}+1$ of the corresponding matrices $A_{k}$, $k=1,2,,\ldots ,\lfloor 
\frac{p_{0}}{2}\rfloor $. \newline
Let $X$ be an $n\times p_{0}$ matrix whose columns are nonnegative
eigenvectors of the matrix $A$ in \eqref{Apis}, that is, 
\begin{equation}
X=\left[ 
\begin{array}{cccccc}
\mathbf{x}_{1} &  &  &  &  &  \\ 
& \ddots &  &  &  &  \\ 
&  & \mathbf{x}_{\frac{p_{0}}{2}} &  &  &  \\ 
&  &  & \mathbf{x}_{\frac{p_{0}}{2}} &  &  \\ 
&  &  &  & \ddots &  \\ 
&  &  &  &  & \mathbf{x}_{1}%
\end{array}%
\right] ,\ \text{for even }p_{0},  \label{X}
\end{equation}%
and 
\begin{equation}
X=\left[ 
\begin{array}{ccccccc}
\mathbf{x}_{1} &  &  &  &  &  &  \\ 
& \ddots &  &  &  &  &  \\ 
&  & \mathbf{x}_{\lfloor \frac{p_{0}}{2}\rfloor } &  &  &  &  \\ 
&  &  & \mathbf{y}_{\lceil \frac{p_{0}}{2}\rceil } &  &  &  \\ 
&  &  &  & \mathbf{x}_{\lfloor \frac{p_{0}}{2}\rfloor } &  &  \\ 
&  &  &  &  & \mathbf{\ddots } &  \\ 
&  &  &  &  &  & \mathbf{x}_{1}%
\end{array}%
\right] ,\text{ for odd }p_{0}.  \label{XO}
\end{equation}

\noindent where, $\mathbf{x}_{k}=\mathbf{e}^{T}=[1,1,\ldots ,1]$ is the
Perron eigenvector of $A_{k}$ in (\ref{Apis}), $k=1,2,\ldots ,\frac{p_{0}}{2}
$ $(\lfloor \frac{p_{0}}{2}\rfloor ),$ and $\mathbf{y}_{\lceil \frac{p_{0}}{2%
}\rceil }$ is the Perron eigenvector of $A_{\lceil \frac{p_{0}}{2}\rceil }$,
with $\mathbf{y}_{\lceil \frac{p_{0}}{2}\rceil }=J\mathbf{y}_{\lceil \frac{%
p_{0}}{2}\rceil }\geq 0$ (symmetric vector). Observe that 
\begin{equation*}
\text{if \ }A_{k}\mathbf{e}=\omega _{k}\mathbf{e}\text{ \ then }JA_{k}J%
\mathbf{e}=\omega _{k}\mathbf{e}.
\end{equation*}%
Now we state the main result of this section. First, we consider the even $%
p_{0}$ case:

\begin{theorem}
\label{Ana3}Let $\Lambda =\{\lambda _{1},\lambda _{2},\ldots ,\lambda _{n}\}$
be a list of complex numbers with $\Lambda =\overline{\Lambda },$ $\lambda
_{1}\geq \left\vert \lambda _{i}\right\vert ,$ $i=2,3,\ldots ,n,$ and $%
\sum\limits_{i=1}^{n}\lambda _{i}\geq 0.$ Suppose there exists a partition
of $\Lambda $ (as defined in (\ref{parti})),%
\begin{eqnarray*}
\Lambda &=&\Lambda _{0}\cup \Lambda _{1}\cup \cdots \cup \Lambda _{\frac{%
p_{0}}{2}}\cup \Lambda _{\frac{p_{0}}{2}}\cup \cdots \cup \Lambda _{1},\text{
\ with even }p_{0}, \\
\Lambda _{0} &=&\{\lambda _{01},\lambda _{02},\ldots ,\lambda _{0p_{0}}\},%
\text{ \ }\lambda _{01}=\lambda _{1} \\
\Lambda _{k} &=&\{\lambda _{k1},\lambda _{k2},\ldots ,\lambda _{kp_{k}}\},%
\text{ \ }k=1,2,\ldots ,\frac{p_{0}}{2},\text{ }
\end{eqnarray*}%
where some of the lists $\Lambda _{k}$ can be empty, such that the following
conditions are satisfied:\newline
$i)$ For each $k=1,2,\ldots ,\frac{p_{0}}{2},$ there exists a nonnegative
matrix with spectrum 
\begin{equation*}
\Gamma _{k}=\{\omega _{k},\lambda _{k1},\lambda _{k2},\ldots ,\lambda
_{kp_{k}}\},\text{ }0\leq \omega _{k}\leq \lambda _{1},
\end{equation*}%
$ii)$ There exists a centrosymmetric nonnegative matrix of order $p_{0},$
with spectrum $\Lambda _{0}$ and diagonal entries $\omega _{1},\omega
_{2},\ldots ,\omega _{\frac{p_{0}}{2}},\omega _{\frac{p_{0}}{2}},\ldots
,\omega _{2},\omega _{1}.$\newline
Then $\Lambda $ is realizable by an $n\times n$ centrosymmetric matrix.
\end{theorem}

\begin{proof}
From $i)$ let $A_{k}$ be a $(p_{k}+1)\times (p_{k}+1)$ nonnegative matrix
with spectrum $\Gamma _{k},$ $k=1,2,\ldots ,\frac{p_{0}}{2}.$ We may assume
that $A_{k}\in \mathcal{CS}_{\omega _{k}}.$ Then $A_{k}\mathbf{e}=\omega _{k}%
\mathbf{e},$ The matrix%
\begin{equation*}
A=A_{1}\oplus A_{2}\oplus \cdots \oplus A_{\frac{p_{0}}{2}}\oplus JA_{\frac{%
p_{0}}{2}}J\oplus \cdots \oplus JA_{2}J\oplus JA_{1}J
\end{equation*}%
is centrosymmetric nonnegative with spectrum 
\begin{equation*}
\Gamma =\Gamma _{1}\cup \cdots \cup \Gamma _{\frac{p_{0}}{2}}\cup \Gamma _{%
\frac{p_{0}}{2}}\cup \cdots \cup \Gamma _{1}.
\end{equation*}%
Let $X$ be the matrix in (\ref{X}). Then $X$ is  centrosymmetric nonnegative
of zeros and ones.\newline
Now, from $ii)$ and Theorem \ref{Rado} \ let $B=\Omega +\mathcal{C}X$ centrosymmetric nonnegative matrix of
order $p_{0}$ with spectrum $\Lambda _{0}$ and diagonal entries 
\begin{equation*}
\omega _{1},\ldots ,\omega _{\frac{p_{0}}{2}},\omega _{\frac{p_{0}}{2}%
},\ldots ,\omega _{1}.
\end{equation*}%
Then $\mathcal{C}X=B-\Omega =%
\begin{bmatrix}
b_{1} & \cdots & b_{\frac{p_{0}}{2}} & J_{p_{0}}b_{\frac{p_{0}}{2}} & \cdots
& J_{p_{0}}b_{1}%
\end{bmatrix}%
$ is centrosymmetric nonnegative, where $b_{k}$ and $J_{p_{0}}b_{k}$, $%
k=1,\ldots ,\frac{p_{0}}{2}$, are the columns of $B-\Omega $. Therefore the $%
p_{0}\times n$ matrix $\mathcal{C}$ can be obtained as follows 
\begin{equation}
\mathcal{C}=[b_{1}\underbrace{0\ldots 0}_{p_{1}-times}\cdots b_{\frac{p_{0}}{2}}%
\underbrace{0\ldots 0}_{\frac{p_{0}}{2}-times}\underbrace{0\ldots 0}_{\frac{%
p_{0}}{2}-times}J_{p_{0}}b_{\frac{p_{0}}{2}}\cdots \underbrace{0\ldots 0}%
_{p_{1}-times}J_{p_{0}}b_{1}].  \label{cma}
\end{equation}%
It is easy to see that $\mathcal{C}$ is centrosymmetric nonnegative. Then the
perturbation $A+X\mathcal{C}$, with $A,X,\mathcal{C}$ as in \eqref{Apis},\eqref{X},\eqref{cma}
respectively, is nonnegative centrosymmetric and by Theorem \ref{Rado}, $A+X\mathcal{C}$
have the spectrum $\Lambda $.
\end{proof}\ \\ \\ 
Now we consider the odd $p_{0}$ case:
\begin{theorem}
\label{Ana4}Let $\Lambda =\{\lambda _{1},\lambda _{2},\ldots ,\lambda _{n}\}$
be a list of complex numbers with $\Lambda =\overline{\Lambda },$ $\lambda
_{1}\geq \left\vert \lambda _{i}\right\vert ,$ $i=2,3,\ldots ,n,$ and $%
\sum\limits_{i=1}^{n}\lambda _{i}\geq 0.$ Suppose there exists a partition
of $\Lambda $ (as defined in (\ref{parti})),%
\begin{eqnarray*}
\Lambda &=&\Lambda _{0}\cup \Lambda _{1}\cup \cdots \cup \Lambda _{\lfloor 
\frac{p_{0}}{2}\rfloor }\cup \Lambda _{\lceil \frac{p_{0}}{2}\rceil }\cup
\Lambda _{\lfloor \frac{p_{0}}{2}\rfloor }\cup \cdots \cup \Lambda _{1},%
\text{with odd}\ p_{0},  \\
\Lambda _{0} &=&\{\lambda _{01},\lambda _{02},\ldots ,\lambda _{0p_{0}}\},%
\text{ \ }\lambda _{01}=\lambda _{1} \\
\Lambda _{k} &=&\{\lambda _{k1},\lambda _{k2},\ldots ,\lambda _{kp_{k}}\},%
\text{ \ }k=1,2,\ldots ,\bigg\lceil\frac{p_{0}}{2}\bigg\rceil,\text{ }
\end{eqnarray*}%
where some of the lists $\Lambda _{k}$ can be empty, such that the following
conditions are satisfied:\newline
$i)$ For each $k=1,2,\ldots ,\lfloor \frac{p_{0}}{2}\rfloor ,$ there exists
a nonnegative matrix with spectrum 
\begin{equation*}
\Gamma _{k}=\{\omega _{k},\lambda _{k1},\lambda _{k2},\ldots ,\lambda
_{kp_{k}}\},\text{ }0\leq \omega _{k}\leq \lambda _{1},
\end{equation*}%
while for $k=\lceil \frac{p_{0}}{2}\rceil $ there exists a centrosymmetric nonnegative
matrix with spectrum $\Gamma _{\lceil \frac{p_{0}}{2}\rceil
} $. \newline
$ii)$ There exists a centrosymmetric nonnegative matrix of order $p_{0},$
with spectrum $\Lambda _{0}$ and diagonal entries $\omega _{1},\omega
_{2},\ldots ,\omega _{\lfloor \frac{p_{0}}{2}\rfloor },\omega _{\lceil \frac{%
p_{0}}{2}\rceil },\omega _{\lfloor \frac{p_{0}}{2}\rfloor },\ldots ,\omega
_{2},\omega _{1}.$ \newline
Then $\Lambda $ is realizable by an $n\times n$ centrosymmetric matrix.
\end{theorem}
\begin{proof}
Case 1. The list $\Gamma _{\lceil \frac{p_{0}}{2}\rceil }$ is realizable by
a centrosymmetric nonnegative matrix of odd order: In this case the proof
follows as the proof of Theorem \ref{Ana3}, with a perturbation of the form $A+X\mathcal{C}$, with $A,X$ as in \eqref{Apis}, \eqref{XO} respectively, and   
\begin{eqnarray*}
\mathcal{C} &=&[b_{1}\underbrace{0\cdots 0}_{p_{1}-times}\cdots b_{\lfloor \frac{p_{0}%
}{2}\rfloor }\underbrace{0\cdots 0}_{\lfloor \frac{p_{0}}{2}\rfloor -times}%
\underbrace{0\cdots 0}_{\frac{\lceil \frac{p_{0}}{2}\rceil }{2}%
-times}b_{\lceil \frac{p_{0}}{2}\rceil }\underbrace{0\cdots 0}_{\frac{\lceil 
\frac{p_{0}}{2}\rceil }{2}-times} \\
&&\underbrace{0\cdots 0}_{\lfloor \frac{p_{0}}{2}\rfloor
-times}J_{p_{0}}b_{\lfloor \frac{p_{0}}{2}\rfloor }\cdots \underbrace{%
0\cdots 0}_{p_{1}-times}J_{p_{0}}b_{1}].
\end{eqnarray*}%
Case 2. The list $\Gamma _{\lceil \frac{p_{0}}{2}\rceil }$ is realizable by
a centrosymmetric nonnegative matrix of order even: In this case we consider 
$A$ as in \eqref{Apis}, $X$ as in \eqref{XO} with all its columns
normalized, and in $ii)$ we take $B=\Omega +\mathcal{C}$. Then $\mathcal{C}=B-\Omega $, is
 centrosymmetric nonnegative, and by Lemma \ref{Le1} and Theorem \ref{Rado}, the perturbation $A+X\mathcal{C}X^{T}$ is centrosymmetric nonnegative with
spectrum $\Lambda $.
\end{proof}\ \\ \\
To apply Theorem \ref{Ana3} and Theorem \ref{Ana4}, we need to know
conditions under which there exists a $p_{0}\times p_{0}$ centrosymmetric
nonnegative matrix with spectrum $\Lambda _{0}$ and diagonal entries%
\begin{eqnarray*}
&&\omega _{1},\omega _{2},\ldots ,\omega _{\frac{p_{0}}{2}},\omega _{\frac{%
p_{0}}{2}},\ldots ,\omega _{2},\omega _{1},\text{ for even }p_{0},\text{ \ or%
} \\
&&\omega _{1},\omega _{2},\ldots ,\omega _{\lfloor \frac{p_{0}}{2}\rfloor
},\omega _{\lceil \frac{p_{0}}{2}\rceil },\omega _{\lfloor \frac{p_{0}}{2}%
\rfloor },\ldots ,\omega _{2},\omega _{1}\text{ \ for odd }p_{0}.
\end{eqnarray*}%
\ \newline
In \cite[Theorem $3.1$]{Somph} (see also \cite{Somphot}) it has been proved
that the spectrum of a $3\times 3$ nonnegative centrosymmetric matrix is
real. Then, since bisymmetric matrices are centrosymmetric, we may use the
necessary and sufficient conditions given in \cite{Julio}, for the existence
of a $3\times 3$ centrosymmetric nonnegative matrix with real eigenvalues $%
\lambda _{1},\lambda _{2},\lambda _{3}$ and diagonal entries $\omega
_{1},\omega _{2},\omega _{1}.$ For $n=4$ we have, in the Theorems \ref{Sule
Real} and \ref{Sule Com} below, sufficient conditions for the existence and construction of a centrosymmetric nonnegative matrix with prescribed eigenvalues and diagonal
entries.\\ \\
The following result show that for lists of four real numbers, both
problems, the nonnegative inverse eigenvalue problem and the centrosymmetric
nonnegative inverse eigenvalue problem, are equivalent.
\begin{theorem}
Let $\Lambda=\{\lambda_{1},\lambda_{2},\lambda_{3},\lambda_{4}\}$ be a list
realizable of real numbers. Then $\Lambda$ is in particular realizable by a
centrosymmetric nonnegative matrix.
\end{theorem}
\begin{proof}
Case 1: If $\lambda _{1}\geq \lambda _{2}\geq \lambda _{3}\geq \lambda
_{4}\geq 0,$ the result follows from Theorem \ref{th. real} when $n=4$. 
\newline
Case 2: If $\lambda _{1}>0\geq \lambda _{2}\geq \lambda _{3}\geq \lambda
_{4} $, that is, if the list is of real Suleimanova type, then the
result follows from Remark \ref{R2}. \newline
Case 3: If $\lambda _{1}\geq \lambda _{2}\geq \lambda _{3}\geq 0>\lambda
_{4} $, we have the following general result:\\
A list realizable of real numbers with only one negative eigenvalue is also
realizable by a centrosymmetric nonnegative  matrix. In fact, let $\Lambda
=\{\lambda _{1},\ldots ,\lambda _{n}\}$, with $\lambda _{1}\geq \lambda
_{2}\geq \cdots \geq \lambda _{n-1}\geq 0>\lambda _{n}$. \\
If $n=2m$ is even we take the partition $\Lambda =\Lambda _{1}\cup \Lambda
_{2}$, with 
\begin{equation*}
\Lambda _{1}=\{\lambda _{1},\ldots ,\lambda _{m}\},\ \ \Lambda
_{2}=\{\lambda _{m+1},\ldots ,\lambda _{n}\},
\end{equation*}%
and we set 
\begin{equation*}
A+JB=diag\{\lambda _{1},\lambda _{2},\ldots ,\lambda _{m}\},\
A-JB=diag\{\lambda _{n},\lambda _{m+1},\ldots ,\lambda _{n-1}\},
\end{equation*}%
and the proof follows as before. \newline
If $n=2m+1$ is odd we take the partition $\Lambda =\Lambda _{1}\cup \Lambda
_{2}$, with 
\begin{equation*}
\Lambda _{1}=\{\lambda _{1},\ldots ,\lambda _{\lceil \frac{n}{2}\rceil }\},\
\ \Lambda _{2}=\{\lambda _{\lceil \frac{n}{2}\rceil +1}\ldots ,\lambda
_{n}\}.
\end{equation*}%
and we set 
\begin{equation*}
\begin{bmatrix}
A+JB & \sqrt{2}\mathbf{x} \\ 
\sqrt{2}\mathbf{y}^{T} & c%
\end{bmatrix}%
=diag\{\lambda _{1},\lambda _{2},\ldots ,\lambda _{\lfloor \frac{n}{2}%
\rfloor },\lambda _{\lceil \frac{n}{2}\rceil }\},
\end{equation*}%
\begin{equation*}
A-JB=diag\{\lambda _{n},\lambda _{\lceil \frac{n}{2}\rceil +1},\ldots
,\lambda _{n-1}\}.
\end{equation*}%
Then, the proof follows as before. \newline
Case 4: If $\lambda _{1}\geq \lambda _{2}\geq 0>\lambda _{3}\geq \lambda
_{4} $, with $\lambda _{2}+\lambda _{3}>0$; \\\
we take the partition $\Lambda =\Lambda _{1}\cup \Lambda _{2}$, with $%
\Lambda _{1}=\{\lambda _{1},\lambda _{2}\}$, $\Lambda _{2}=\{\lambda
_{3},\lambda _{4}\},$ and we set 
\begin{equation*}
A+JB=diag\{\lambda _{1},\lambda _{2}\},\ \ A-JB=diag\{\lambda _{4},\lambda
_{3}\}.
\end{equation*}%
Then, the proof follows as before.\newline
If $\lambda _{1}\geq \lambda _{2}\geq 0>\lambda _{3}\geq \lambda _{4}$, with 
$\lambda _{2}+\lambda _{3}<0;$ \newline
we take the partition $\Lambda _{1}=\{\lambda _{1},\lambda _{2}\},\ \
\Lambda _{2}=\{\lambda _{3},\lambda _{4}\}$ and we set 
\begin{equation*}
A+JB=%
\begin{bmatrix}
-\lambda _{3} & 1 \\ 
(-\lambda _{3})(\lambda _{1}+\lambda _{2}+\lambda _{3})-\lambda _{1}\lambda
_{2} & \lambda _{1}+\lambda _{2}+\lambda _{3}%
\end{bmatrix}%
\end{equation*}%
\begin{equation*}
A-JB=diag\{\lambda _{3},\lambda _{4}\}.
\end{equation*}%
Then 
\begin{equation*}
A=\frac{1}{2}%
\begin{bmatrix}
0 & 1 \\ 
-(\lambda _{2}+\lambda _{3})(\lambda _{1}+\lambda _{3}) & \lambda
_{1}+\lambda _{2}+\lambda _{3}+\lambda _{4}%
\end{bmatrix}%
,
\end{equation*}
\begin{equation*}
B=\frac{1}{2}J%
\begin{bmatrix}
-2\lambda _{3} & 1 \\ 
-(\lambda _{2}+\lambda _{3})(\lambda _{1}+\lambda _{3}) & \lambda
_{1}+\lambda _{2}+\lambda _{3}-\lambda _{4}%
\end{bmatrix}%
\end{equation*}
are both nonnegative matrices, and 
\begin{equation*}
C=%
\begin{bmatrix}
A & JBJ \\ 
B & JAJ%
\end{bmatrix}%
,
\end{equation*}%
is a $4\times 4$ centrosymmetric nonnegative matrix.
\end{proof}\ \\ \\
Now we give a sufficient condition for the existence and construction of a $%
4\times 4$ centrosymmetric nonnegative matrix with prescribed eigenvalues
and diagonal entries. We start with lists of real numbers.
\begin{theorem}
\label{Sule Real} Let $\Lambda=\{\lambda_{1},\lambda_{2},\lambda_{3},%
\lambda_{4}\}$ be a list of real numbers with $\sum_{i=1}^{4}\lambda_{i}\geq
0$, $\lambda_{1}\geq \vert\lambda_{j}\vert$, $j=2,3,4$. Let $%
\{\omega_{1},\omega_{2},\omega_{2},\omega_{1}\}$ be a list of nonnegative
numbers such that the following conditions are satisfied: \\
$i)$ $0\leq\omega_{k}\leq \lambda_{1}$\ $k=1,2$,\\
$ii)$ $\omega_{1}+\omega_{2}=\frac{1}{2}(\lambda_{1}+\lambda_{2}+%
\lambda_{3}+\lambda_{4})$, \newline
$iii)$ $\omega_{k}\geq\lambda_{k+2}$ \ $k=1,2$,\\
$iv)$ $(2\omega_{1}-\lambda_{3})(2\omega_{2}-\lambda_{4})\geq\lambda_{1}%
\lambda_{2}$. \\
Then, there exists a centrosymmetric nonnegative  matrix with espectrum $%
\Lambda$ and diagonal entries $\omega_{1},\omega_{2},\omega_{2},\omega_{1}$.
\end{theorem}
\begin{proof}
Let $\Lambda $ be particioned as $\Lambda =\Lambda _{1}\cup \Lambda _{2}$,
with $\Lambda _{1}=\{\lambda _{1},\lambda _{2}\}$ and $\Lambda
_{2}=\{\lambda _{3},\lambda _{4}\}$.We set 
\begin{equation*}
A+JB=%
\begin{bmatrix}
2\omega _{1}-\lambda _{3} & 1 \\ 
(2\omega _{1}-\lambda _{3})(2\omega _{2}-\lambda _{4})-\lambda _{1}\lambda
_{2} & 2\omega _{2}-\lambda _{4}%
\end{bmatrix}%
,\ A-JB=diag\{\lambda _{3},\lambda _{4}\},
\end{equation*}%
with spectrum $\Lambda _{1}$ and $\Lambda _{2}$ respectively. Then 
\begin{equation*}
A=%
\begin{bmatrix}
\omega _{1} & \frac{1}{2} \\ 
\frac{1}{2}((2\omega _{1}-\lambda _{3})(2\omega _{2}-\lambda _{4})-\lambda
_{1}\lambda _{2}) & \omega _{2}%
\end{bmatrix}%
,
\end{equation*}%
\begin{equation*}
B=%
\begin{bmatrix}
\frac{1}{2}((2\omega _{1}-\lambda _{3})(2\omega _{2}-\lambda _{4})-\lambda
_{1}\lambda _{2}) & \omega _{2}-\lambda _{4} \\ 
\omega _{1}-\lambda _{3} & 1%
\end{bmatrix}%
\end{equation*}%
are both nonnegative matrices and 
\begin{equation*}
C=%
\begin{bmatrix}
A & JBJ \\ 
B & JAJ%
\end{bmatrix}%
\end{equation*}%
is centrosymmetric nonnegative matrix with spectrum $\Lambda $ and diagonal
entries $\omega _{1},\omega _{2},\omega _{2},\omega _{1}$.
\end{proof}

\ \newline
Note that in the above theorem, if $\Lambda $ is a list of real numbers of
Suleimanova type, then the conditions $iii)$ and $iv)$ are always satisfied.
Then $i)$ and $ii)$ become necessary and sufficient conditions.\newline
\ \newline
In \cite[Theorem 3.1]{Somph} the authors show that $\lambda _{1}+\lambda
_{2}-2|a|\geq 0$ and $\lambda _{1}-\lambda _{2}-2b\geq 0$ are necessary and
sufficient conditions for the list $\{\lambda _{1},\lambda _{2},a+ib,a-ib\}$,
with $\lambda _{1},\lambda _{2},a\in \mathbb{R},b>0$ to be the spectrum of a normal centrosymmetric nonnegative matrix. Now, we
give a sufficient condition for the existence and construction of a centrosymmetric nonnegative matrix with prescribed spectrum and diagonal
entries.

\begin{theorem}
\label{Sule Com} Let $\Lambda =\{\lambda _{1},\lambda _{2},a+ib,a-ib\}$ be a
list of complex numbers with $\lambda _{1},\lambda _{2},a\in \mathbb{R}$, $%
b>0$, $\lambda_{1}+\lambda_{2}-2\vert a\vert\geq 0$ and $\lambda_{1}-%
\lambda_{2}-2b\geq0$. Let $\{\omega_{1},\omega_{2},\omega_{2},\omega_{1}\}$
be a list of nonnegative numbers, such that the following conditions
are satisfied: \\
$i)$ $0\leq \omega_{k}\leq \lambda_{1}$, $k=1,2$, \\
$ii)$ $\frac{1}{2}(\lambda_{1}+\lambda_{2}+2a)=\omega_{1}+\omega_{2}$, 
\\
$iii)$ $(2\omega_{1}-a)(2\omega_{2}-a)\geq \lambda_{1}\lambda_{2}+b^{2}$, 
\\
$iv)$ $\omega_{k}\geq a$, $k=1,2$.\\
Then, there exists a centrosymmetric nonnegative matrix with spectrum $%
\Lambda$ and diagonal entries $\omega_{1},\omega_{2},\omega_{2},\omega_{1}$.
\end{theorem}
\begin{proof}
Let $\Lambda $ be particioned as $\Lambda =\Lambda _{1}\cup \Lambda _{2}$
with $\Lambda _{1}=\{\lambda _{1},\lambda _{2}\}$, $\Lambda
_{2}=\{a+ib,a-ib\}$. We set 
\begin{equation*}
A+JB=%
\begin{bmatrix}
2\omega _{1}-a & 1 \\ 
(2\omega _{1}-a)(2\omega _{2}-a)-\lambda _{1}\lambda _{2} & 2\omega _{2}-a%
\end{bmatrix}%
,\ A-JB=%
\begin{bmatrix}
a & 1 \\ 
-b^{2} & a%
\end{bmatrix}%
,
\end{equation*}%
with spectrum $\Lambda _{1}$ and $\Lambda _{2}$, respectively. Then 
\begin{equation*}
A=%
\begin{bmatrix}
\omega _{1} & 1 \\ 
\frac{1}{2}((2\omega _{1}-a)(2\omega _{2}-a)-\lambda _{1}\lambda _{2}-b^{2})
& \omega _{2}%
\end{bmatrix}%
,
\end{equation*}%
\begin{equation*}
B=%
\begin{bmatrix}
\frac{1}{2}((2\omega _{1}-a)(2\omega _{2}-a)-\lambda _{1}\lambda _{2}+b^{2})
& \omega _{2}-a \\ 
\omega _{1}-a & 0%
\end{bmatrix}%
,
\end{equation*}%
are both nonnegative matrices, and 
\begin{equation*}
C=%
\begin{bmatrix}
A & JBJ \\ 
B & JAJ%
\end{bmatrix}%
,
\end{equation*}%
is nonnegative centrosymmetric matrix with spectrum $\Lambda $ and diagonal
entries $\omega _{1},\omega _{2},\omega _{2},\omega _{1}$.
\end{proof}\ \\ \\
Note that in the above theorem, if $\Lambda $ is a list of complex numbers
of Suleimanova type, then the conditions $\lambda _{1}+\lambda _{2}-2|a|\geq 0$ and 
$\lambda _{1}-\lambda _{2}-2b\geq 0$ are always satisfied. Hence, for $n=4,$
a list of complex numbers of Suleimanova type is always realizable by a
centrosymmetric matrix. Moreover, conditions $iii)$ and $iv)$ are always
satisfied for this kind of lists. Then, conditions $i)$ and $ii)$ become
necessary and sufficient conditions.
\section{Examples}
\begin{example}
Let $\Lambda =\{20,-1,-2,-3,-2+2i,-2-2i,-3+i,-3-i,-1+i,-1-i\}$. We shall use Corollary \ref{coro1} to
construct a centrosymmetric nonnegative matrix with spectrum $\Lambda $.
Then we consider the list%
\begin{equation*} 
\Lambda ^{\prime }=\{18,-1,-2,-3,-2+2i,-2-2i,-3+i,-3-i,-1+i,-1-i\},
\end{equation*}%
partitioned as $\Lambda ^{\prime }=\Lambda _{1}\cup \Lambda _{2}$
with $\Lambda _{1}=\{18,-1,-2,-2+2i,-2-2i\}$, $\Lambda
_{2}=\{-3,-3+i,-3-i,-1+i,-1-i\}.$ Then 
\begin{equation*}
A-JB=%
\begin{bmatrix}
-3 &  &  &  &  \\ 
& -3 & -1 &  &  \\ 
& 1 & -3 &  &  \\ 
&  &  & -1 & -1 \\ 
&  &  & 1 & -1%
\end{bmatrix}%
\end{equation*}%
has spectrum $\Lambda _{2}$. Now, from Lemma \ref{SJC}, we construct the
nonnegative matrix 
\begin{eqnarray*}
A+JB &=&
\begin{bmatrix}
0 &1  &2  &2  &2  \\ 
1 & 0 &2  &2  &2  \\ 
2 &1  &0 &2  &2  \\ 
4 &1  &2  &0 & 0 \\ 
0 &1  &2  &4 &0
\end{bmatrix}
+%
\begin{bmatrix}
1 \\ 
1 \\ 
1 \\ 
1 \\ 
1%
\end{bmatrix}%
\begin{bmatrix}
3 & 3 & 3 & 1 & 1
\end{bmatrix}
\\
&=&%
\begin{bmatrix}
3 & 4 & 5 & 3 & 3 \\ 
4 & 3 & 5 & 3 & 3 \\ 
5 & 4 & 3 & 3 & 3 \\ 
7 & 4 & 5 & 1 & 1 \\ 
3 & 4 & 5 & 5 & 1%
\end{bmatrix}
,
\end{eqnarray*}
with spectrum $\Lambda _{1}$ and diagonal entries $3,3,3,1,1.$ Then 
\begin{equation*}
A=\frac{1}{2}%
\begin{bmatrix}
0 & 4 & 5 & 3 & 3 \\ 
4 & 0 & 4 & 3 & 3 \\ 
5 & 5 & 0 & 3 & 3 \\ 
7 & 4 & 5 & 0 & 0 \\ 
3 & 4 & 5 & 6 & 0%
\end{bmatrix}%
,\text{ }B=\frac{1}{2}%
\begin{bmatrix}
3 & 4 & 5 & 4 & 2 \\ 
7 & 4 & 5 & 2 & 2 \\ 
5 & 3 & 6 & 3 & 3 \\ 
4 & 6 & 6 & 3 & 3 \\ 
6 & 4 & 5 & 3 & 3%
\end{bmatrix}%
,
\end{equation*}%
and%
\begin{equation*}
C^{\prime }=\frac{1}{2}%
\begin{bmatrix}
0 & 4 & 5 & 3 & 3 & 3 & 3 & 5 & 4 & 6 \\ 
4 & 0 & 4 & 3 & 3 & 3 & 3 & 6 & 6 & 4 \\ 
5 & 5 & 0 & 3 & 3 & 3 & 3 & 6 & 3 & 5 \\ 
7 & 4 & 5 & 0 & 0 & 2 & 2 & 5 & 4 & 7 \\ 
3 & 4 & 5 & 6 & 0 & 2 & 4 & 5 & 4 & 3 \\ 
3 & 4 & 5 & 4 & 2 & 0 & 6 & 5 & 4 & 3 \\ 
7 & 4 & 5 & 2 & 2 & 0 & 0 & 5 & 4 & 7 \\ 
5 & 3 & 6 & 3 & 3 & 3 & 3 & 0 & 5 & 5 \\ 
4 & 6 & 6 & 3 & 3 & 3 & 3 & 4 & 0 & 4 \\ 
6 & 4 & 5 & 3 & 3 & 3 & 3 & 5 & 4 & 0%
\end{bmatrix}%
\end{equation*}%
is  centrosymmetric nonnegative with spectrum $\Lambda ^{\prime }.$ Then 
\begin{equation*}
C=C^{\prime }+\frac{2}{\mathbf{e}^{T}\mathbf{e}}\mathbf{ee}^{T}=\frac{1}{10}%
\begin{bmatrix}
2 & 22 & 27 & 17 & 17 & 17 & 17 & 27 & 22 & 32 \\ 
22 & 2 & 22 & 17 & 17 & 17 & 17 & 32 & 32 & 22 \\ 
27 & 27 & 2 & 17 & 17 & 17 & 17 & 32 & 17 & 27 \\ 
37 & 22 & 27 & 2 & 2 & 12 & 12 & 27 & 22 & 37 \\ 
17 & 22 & 27 & 32 & 2 & 12 & 22 & 27 & 22 & 17 \\ 
17 & 22 & 27 & 22 & 12 & 2 & 32 & 27 & 22 & 17 \\ 
37 & 22 & 27 & 12 & 12 & 2 & 2 & 27 & 22 & 37 \\ 
27 & 17 & 32 & 17 & 17 & 17 & 17 & 2 & 27 & 27 \\ 
22 & 32 & 32 & 17 & 17 & 17 & 17 & 22 & 2 & 22 \\ 
32 & 22 & 27 & 17 & 17 & 17 & 17 & 27 & 22 & 2%
\end{bmatrix}%
\end{equation*}%
is centrosymmetric nonnegative with spectrum $\Lambda $.
\end{example}

\begin{example}
Let $\Lambda =\{10,3,1+i,1-i,-2+2i,-2-2i,-2+2i,-2-2i\}$. Now we use Theorem %
\ref{Ana3} and Theorem \ref{Sule Com} to construct a centrosymmetric nonnegative matrix with spectrum $%
\Lambda .$ We take the partition 
\begin{equation*}
\Lambda =\Lambda _{0}\cup \Lambda _{1}\cup \Lambda _{2}\cup \Lambda _{2}\cup
\Lambda _{1},\ \ \text{with}
\end{equation*}%
\begin{equation*}
\Lambda _{0}=\{10,3,1+i,1-i\},\ \ \Lambda _{1}=\{-2+2i,-2-2i\},\ \ \Lambda
_{2}=\emptyset,
\end{equation*}
and consider the associated realizable lists 
\begin{equation*}
\Gamma _{1}=\{4,-2+2i,-2-2i\},\ \ \Gamma _{2}=\{\frac{7}{2}\},
\end{equation*}
with realizing matrices \[A_{1}=\begin{bmatrix}0 & 0 & 4 \\ 
2 & 0 & 2  \\ 
0 & 4 & 0 \end{bmatrix}, \ A_{2}=\begin{bmatrix}\frac{7}{2}\end{bmatrix}=JA_{2}J, \ \ JA_{1}J=\begin{bmatrix}0 & 4 & 0 \\ 
 2 & 0 & 2 \\ 
 4 & 0 & 0\end{bmatrix}.\] Then
\begin{equation*}
A=%
\begin{bmatrix}
0 & 0 & 4 &  &  &  &  &  \\ 
2 & 0 & 2 &  &  &  &  &  \\ 
0 & 4 & 0 &  &  &  &  &  \\ 
&  &  & \frac{7}{2} &  &  &  &  \\ 
&  &  &  & \frac{7}{2} &  &  &  \\ 
&  &  &  &  & 0 & 4 & 0 \\ 
&  &  &  &  & 2 & 0 & 2 \\ 
&  &  &  &  & 4 & 0 & 0%
\end{bmatrix} \ \ \text{and} \ \ X=\begin{bmatrix}
1 &  &  &  \\ 
1 &  &  &  \\ 
1 &  &  &  \\ 
& 1 &  &  \\ 
&  & 1 &  \\ 
&  &  & 1 \\ 
&  &  & 1 \\ 
&  &  & 1%
\end{bmatrix}.
\end{equation*}
From Theorem \ref{Sule Com} we compute the centrosymmetric nonnegative matrix with spectrum $\Lambda_{0}$ and diagonal entries $\{4,\frac{7}{2},\frac{7}{2},4\}$
\begin{equation*}
B=
\begin{bmatrix}
4 & 1 & 0 & 3 \\ 
\frac{11}{2} & \frac{7}{2} & \frac{5}{2} & \frac{13}{2} \\ 
\frac{13}{2} & \frac{5}{2} & \frac{7}{2} & \frac{11}{2} \\ 
3 & 0 & 1 & 4
\end{bmatrix}.
\end{equation*}
Then \[C=\begin{bmatrix}
0 & 0 & 0 & 1 & 0 & 0 & 0 & 3 \\ 
\frac{11}{2} & 0 & 0 & 0 & \frac{5}{2} & 0 & 0 & \frac{13}{2} \\ 
\frac{13}{2} & 0 & 0 & \frac{5}{2} & 0 & 0 & 0 & \frac{11}{2} \\ 
3 & 0 & 0 & 0 & 1 & 0 & 0 & 0%
\end{bmatrix}\] is centrosymmetric nonnegative.
Therefore by Theorem \ref{Ana3}
\begin{align*} 
A+XC =
\begin{bmatrix}
0 & 0 & 4 & 1 & 0 & 0 & 0 & 3 \\ 
2 & 0 & 2 & 1 & 0 & 0 & 0 & 3 \\ 
0 & 4 & 0 & 1 & 0 & 0 & 0 & 3 \\ 
\frac{11}{2} & 0 & 0 & \frac{7}{2} & \frac{5}{2} & 0 & 0 & \frac{13}{2} \\ 
\frac{13}{2} & 0 & 0 & \frac{5}{2} & \frac{7}{2} & 0 & 0 & \frac{11}{2} \\ 
3 & 0 & 0 & 0 & 1 & 0 & 4 & 0 \\ 
3 & 0 & 0 & 0 & 1 & 2 & 0 & 2 \\ 
3 & 0 & 0 & 0 & 1 & 4 & 0 & 0%
\end{bmatrix}%
\text{ is the desired matrix.}
\end{align*}
\end{example}

\end{document}